\documentclass[12pt]{article}
\usepackage{amssymb,amsmath,amsfonts,amsthm,amsxtra,setspace}
\usepackage{indentfirst}

\usepackage{bm}

\newcommand{\w}{\omega}

\newcommand{\mr}{\mathrm}

\newcommand{\mc}{\mathcal}
\newcommand{\bsigma}{\mathbf{\mathop{\pmb{\sum}}}}
\newcommand{\set}[1]{\{#1\}}

\newcommand{\PPI}{\mathbf{PPI}}

\newcommand{\noopsort}[1]{}

\newtheorem{theorem}{Theorem}[section]
\newtheorem{lemma}[theorem]{Lemma}
\newtheorem{eg}{Example}

\newtheorem{cor}[theorem]{Corollary}
\newenvironment{Proof}{\begin{trivlist} \item[] {\bf Proof.}}{\hspace*{0pt}\hfill\qedsymbol\end{trivlist}}

\newsavebox{\Prfref}

\newsavebox{\prfref}

\newtheoremstyle{ref}
{\topsep}	
{\topsep}	
{\it}
{}
{}
{}
{ }
{\thmname{{\bfseries#1}}\thmnumber{ \textbf{#2\thmnote{\rm #3}\textbf .}}}

\theoremstyle{ref}
\newtheorem{lem}[theorem]{Lemma}
\newtheorem{thm}[theorem]{Theorem}
\newtheorem{prop}[theorem]{Proposition}

\newtheorem{prob}{Problem}
\theoremstyle{nnref}
\newtheorem*{defn}{Definition}

\begin{document}
\title{Normality versus paracompactness in locally compact spaces}
\author{Alan Dow{$^1$} and Franklin D. Tall{$^2$}}

\footnotetext[1]{Research supported by NSF grant DMS-1501506.}
\footnotetext[2]{Research supported by NSERC grant A-7354.\vspace*{2pt}}
\date{\today}
\maketitle

\begin{abstract}
This note provides a correct proof of the result claimed by the second author that locally compact normal spaces are collectionwise Hausdorff in certain models obtained by forcing with a coherent Souslin tree. A novel feature of the proof is the use of saturation of the non-stationary ideal on $\omega_1$, as well as of a strong form of Chang's Conjecture. Together with other improvements, this enables the consistent characterization of locally compact hereditarily paracompact spaces as those locally compact, hereditarily normal spaces that do not include a copy of $\omega_1$.
\end{abstract}

\renewcommand{\thefootnote}{}
\footnote
{\parbox[1.8em]{\linewidth}{$2000$ Math.\ Subj.\ Class.\ Primary 54A35, 54D20, 54D45, 03E35; Secondary 03E50, 03E55, 03E57.}\vspace*{5pt}}
\renewcommand{\thefootnote}{}
\footnote
{\parbox[1.8em]{\linewidth}{Key words and phrases: normal, paracompact, locally compact, countably tight,
collectionwise Hausdorff, forcing with a coherent Souslin tree, Martin's Maximum,
PFA(S)[S], {Axiom R}, Moving Off Property.}}

\section{Introduction}
The space of countable ordinals is locally compact, normal, but not paracompact. The question of what additional conditions make a locally compact normal space paracompact has a long history.  At least 45 years ago, it was recognized that subparacompactness plus collectionwise Hausdorffness would do (see e.g. \cite{T1}), as would perfect normality plus metacompactness \cite{A}.  Z. Balogh proved a variety of results under MA$_{\omega_1}$ \cite{B1} and \textbf{Axiom R} \cite{B2}, and was the first to realize the importance of not including a perfect pre-image of $\omega_1$ (equivalently, the one-point compactification being countably tight \cite{B1}).  However, he assumed collectionwise Hausdorffness in order to obtain paracompactness.  A breakthrough came with S. Watson's proof that:

\begin{prop}[{~\cite{W}}]
$V = L$ implies locally compact normal spaces are collectionwise Hausdorff, and hence locally compact normal metacompact spaces are paracompact.
\end{prop}

Watson's proof crucially involved the idea of \emph{character reduction}: if one wants to separate a closed discrete subspace of size $\kappa$, $\kappa$ regular, in a locally compact normal space, it suffices to separate $\kappa$ compact sets, each with an \textit{outer base} of size $\leq \kappa$.

\begin{defn}
An {\normalfont\textbf{outer base}} for a set $K \subseteq X$ is a collection $\mc{B}$ of open sets including $K$ such that each open set including $K$ includes a member of $\mc{B}$.
\end{defn}

The use of $V = L$ was to get that normal spaces of character $\leq \aleph_1$ are collectionwise Hausdorff \cite{F}, and variations on that theme.

It was known that locally compact normal non-collectionwise Hausdorff spaces could be constructed from MA$_{\omega_1}$, indeed from the existence of a $Q$-set \cite{T1}, so it was a big surprise when G. Gruenhage and P. Koszmider proved that:

\begin{prop}[{~\cite{GK}}]
MA$_{\omega_1}$ implies locally compact, normal, metacompact spaces are $\aleph_1$-collectionwise Hausdorff and (hence) paracompact.
\end{prop}

The next result involving iteration axioms and a positive ``normal implies collectionwise Hausdorff" type of result was:

\begin{prop}[{~\cite{LTo}}]
Let $S$ be a coherent Souslin tree (obtainable from $\diamondsuit$ or a Cohen real).  Force MA$_{\omega_1}(S)$, i.e. MA$_{\omega_1}$ for countable chain condition posets preserving $S$.  Then force with $S$.  In the resulting model, there are no first countable $L$-spaces, no compact first countable $S$-spaces, and separable normal first countable spaces are collectionwise Hausdorff.
\end{prop}

The first two statements are consequences of MA$_{\omega_1}$ \cite{Sz}; the last of $V=L$, indeed of $2^{\aleph_0} < 2^{\aleph_1}$.  Larson and Todorcevic used this combination to solve \emph{Kat\v{e}tov's problem}.  This idea of combining consequences of a iteration axiom with ``normal implies collectionwise Hausdorff" consequences of $V = L$ was exploited in \cite{LT1} in order to prove the consistency, modulo a supercompact cardinal, of \emph{every locally compact perfectly normal space is paracompact}.  The large cardinal was later removed, so that:

\begin{thm}[{~\cite{DT2}}]
If ZFC is consistent, then so is ZFC plus every locally compact perfectly normal space is paracompact.
\end{thm}

In the models of \cite{LT1} and \cite{DT2}, every first countable
normal space is collectionwise Hausdorff.  This is achieved in two
stages.  The novel one is: 

\begin{lem}[{~\cite{LT1}}]\label{lem15}
Force with a Souslin tree.  Then\label{LT1}
 normal first countable spaces are
$\aleph_1$-collectionwise Hausdorff. 
\end{lem}

This is obtained by showing that if a normal first countable space is not $\aleph_1$-collectionwise Hausdorff, a generic branch of the Souslin tree induces a generic partition of the unseparated closed discrete subspace which cannot be ``normalized", i.e. there do not exist disjoint open sets about the two halves of the partition.  The argument is a blend of the two usual methods of proving ``normal implies $\aleph_1$-collectionwise Hausdorff" results, namely those of adjoining Cohen subsets of $\omega_1$ by countably closed forcing \cite{T1}, \cite{T2} and using \emph{$\diamondsuit$ for stationary systems on $\omega_1$}, a strengthening of $\diamondsuit$ that holds in $L$ \cite{F}.  It is noteworthy that:

\begin{prop}[{~\cite{T1}, \cite{F}}]
Either force to add $\aleph_2$ Cohen subsets of $\omega_1$, or assume $\diamondsuit$ for stationary subsets of $\omega_1$.  Then normal spaces of character $\leq \aleph_1$ are $\aleph_1$-collectionwise Hausdorff.
\end{prop}

Once one has normal first countable spaces are $\aleph_1$-collectionwise Hausdorff, it is easy to obtain full collectionwise Hausdorffness by starting with $L$ as the ground model and following \cite{F}.  However, if a supercompact cardinal is involved, instead of $L$ we need to follow the method of \cite{LT1}, based on \cite{T2}.  Namely, first make the supercompact indestructible under countably closed forcing \cite{L} and then perform an Easton extension, adding $\kappa^+$ Cohen subsets of each regular $\kappa$, before forcing with the Souslin tree.

In order to extend the theorems about locally compact normal spaces
being paracompact beyond the realm of first countability, one first
needs to get that \textit{locally compact normal spaces are
  collectionwise Hausdorff}.  In \cite{T3}, the second author claimed
to have done so, in the model of \cite{LT1}.  The key was to force to
expand a closed discrete subspace in a locally compact normal space to
a discrete collection of compact sets with countable outer bases and
then apply the methods of \cite{LT1}.  Unfortunately the expansion
argument was flawed.  A corrected argument is presented below, but at
the cost of using a stronger iteration axiom (but not a larger large
cardinal).  

With the conclusion of \cite{T3} restored, \cite{T4}, \cite{LT2}, and
\cite{T} are re-instated.  We shall then proceed to improve the
results of the two latter ones.

\section{PFA$(S)[S]$ and the role of $\omega_1$}

\begin{defn}
	\emph{PFA$(S)$} is the Proper Forcing Axiom (PFA) restricted to those posets that preserve the (Souslinity of the) coherent Souslin tree $S$.\textup{ For the definition of coherence, see e.g. \cite[Chapter 5]{To4}. For a proof that $\diamondsuit$ implies the existence of a coherent Souslin tree, see \cite{Lar3}.}
\end{defn}

\emph{PFA$(S)[S]$ implies $\varphi$} is shorthand for \emph{whenever one forces with a coherent Souslin tree $S$ over a model of PFA$(S)$, $\varphi$ holds.} \emph{$\varphi$ holds in a model of form PFA$(S)[S]$} is shorthand for \emph{there is a coherent Souslin tree $S$ and a model of PFA$(S)$ such that when one forces with $S$ over that model, $\varphi$ holds.}
\newline

For discussion of PFA$(S)[S]$, see \cite{D2}, \cite{To}, \cite{LT1}, \cite{LT2}, \cite{T4}, \cite{T}, \cite{FTT}, \cite{T6}.

The following results appear in \cite{LT2} and \cite{T}, respectively.

\begin{thm}\label{thm:paracompactcopy}
There is a model of form $\mr{PFA}(S)[S]$ in which a locally compact, hereditarily normal space is hereditarily 
paracompact if and only if it does not include a perfect pre-image of $\w_1$.
\end{thm}

\begin{thm}\label{thm:paracompactcountablytight}
There is a model of form $\mr{PFA}(S)[S]$ in which a locally compact normal space is paracompact and countably tight if and 
only if its separable closed subspaces are Lindel\"of and it does not include a perfect pre-image of $\w_1$.
\end{thm}

\begin{defn}
	\textbf{\textup{PPI}} is the assertion that every first countable perfect pre-image of $\omega_1$ includes a copy of $\omega_1$.
\end{defn}

\begin{lem}[{ \cite{DT1}}]
	$\mr{PFA}(S)[S]$ implies \textbf{\textup{PPI}}.
\end{lem}

$\mathbf{PPI}$ was originally proved from PFA in \cite{BDFN}. Using $\mathbf{PPI}$, we are able to weaken ``perfect pre-image" to ``copy" in the improved version of the first theorem, 
but provably cannot in the second theorem.  

\begin{thm}\label{thm:paracompactcopyallmodels}
There is a model of form PFA$(S)[S]$ in which a locally compact, hereditarily normal space is hereditarily paracompact if 
and only if it does not include a copy of $\omega_1$.
\end{thm}

\begin{eg}
There is a locally compact space $X$ (indeed a perfect pre-image of $\omega_1$) which is normal, does not include a copy 
of $\omega_1$, in which all separable closed subspaces are compact, but $X$ is not paracompact.
\end{eg}

It is clear that to establish Theorem \ref{thm:paracompactcopyallmodels}, it suffices to use \ref{thm:paracompactcopy} and apply $\mathbf{PPI}$ after 
proving:

\begin{thm}\label{thm34}
$\mr{PFA}(S)[S]$ implies a hereditarily normal perfect pre-image of $\w_1$ includes a first countable perfect pre-image of 
$\w_1$.
\end{thm}

This follows from:

\begin{lem}
\label{lxb}
Let $X$ be a perfect pre-image of $\omega_1$, and suppose separable subspaces of $X$ are Lindel\"{o}f. Then $X$ includes a 
first countable perfect pre-image of $\omega_1$.
\end{lem}

\noindent and

\begin{lem}[~\cite{To, T}]\label{lxc}
$\mr{PFA}(S)[S]$ implies compact, separable, hereditarily normal spaces are hereditarily Lindel\"{o}f.
\end{lem}

\noindent Here is the proof of Lemma~\ref{lxb}.

\begin{Proof}
Let $f : X\rightarrow\omega_1$, perfect and onto. Then $X$
is locally compact, countably compact, but not compact.  There is a closed $Y \subseteq X$ such that $f' = f|Y$ is 
perfect, irreducible, and maps $Y$ onto $\omega_1$.  So $Y = \bigcup_{\alpha < \omega_1}f'^{-1}(\{\beta : \beta \leq 
\alpha\})$.  Each $D_\alpha = f'^{-1}(\{\beta : \beta \leq \alpha\})$ is clopen and hence countably compact.  It suffices 
to show $D_\alpha$ is hereditarily Lindel\"of, for then points are $G_\delta$ and $D_\alpha$ is first countable.  But then 
$Y$ is first countable, since $D_\alpha$ is open.  To show $D_\alpha$ is hereditarily Lindel\"of, we need only show it is 
separable.  $f_\alpha = f'|D_\alpha$ is irreducible, for if there were a proper closed subset $A$ of $D_\alpha$ such that 
$f'(A) = f'(D_\alpha)$, then $f$ would map $A \cup (Y - D_\alpha)$ onto $\omega_1$, contradicting $f$'s irreducibility.  
But 
\begin{lem}[{~\cite[Section 6.5]{PW}}]
If $f$ is a closed irreducible map of $X$ onto $Y$ and $E$ is dense in $Y$, then $f^{-1}(E)$ is dense in $X$.
\end{lem}
Thus $D_\alpha$ is separable.
\end{Proof}

Let us construct the example that constrains the hoped-for improvement of Theorem \ref{thm:paracompactcountablytight}.  
Consider a stationary, co-stationary subset $E$ of $\omega_1$ and its Stone-\v{C}ech extension $\beta E$.  The identity 
map $\iota$ embeds $E$ into the compact space $\omega_1 + 1$.  $\iota$ extends to $\hat{\iota}$ mapping $\beta E$ onto $
\omega_1+1$; we claim that $\hat{\iota}$ maps only one element -- call it $z$ -- of $\beta E$ to the point $\omega_1$.  
The reason is that every real-valued continuous function on $E$ is eventually constant.  If there were another such point, 
say $z'$, let $f$ be a continuous real-valued function sending $z$ to $0$ and $z'$ to $1$.  Let $U, V$ be open sets about 
the point $\omega_1$ such that $\hat{\iota}^{-1}(U) \subseteq f^{-1}\left(\left[0, \frac{1}{2}\right]\right)$ and $\hat
{\iota}^{-1}(V) \subseteq f^{-1}\left(\left(\frac{1}{2}, 1\right]\right)$.  Then $\hat{\iota}^{-1}(U) \cap \hat{\iota}^{-
1}(V) = \emptyset$, but $U \cap V \cap E$ is cocountable in $E$, contradiction.

Our space $X$ will be $\beta E - \set{z}$.  $\hat{\iota}|X$ maps $X$ onto $\omega_1$; we claim this map is perfect.  By 
3.7.16(iii) of Engelking \cite{E}, it suffices to show that $\hat{\iota}[\beta X - X] = \beta \omega_1 - \omega_1$.  But 
$\beta \omega_1 = \omega_1 + 1$ and $\beta X = \beta E$, so this just says $\hat{\iota}(z) = \omega_1$, which we have.

If $H, K$ are disjoint closed subsets of $X$, then their closures in $\beta E$ have at most $z$ in common.  Thus their 
images $\hat{\iota}[H]$ and $\hat{\iota}[K]$ cannot overlap in a subspace with a point of $E$ in its closure. Since $E$ is stationary, their overlap is countable. Then at least one of them is bounded, and hence compact.  it is then easy to pull back disjoint open sets to establish normality.

For any perfect pre-image of $\omega_1$, it is easy to see that separable closed subspaces are compact, since they are 
included in a pre-image of an initial closed segment of $\omega_1$.

It remains to show that $X$ does not include a copy $W$ of $\omega_1$.  A standard $\beta \mathbb{N}$ argument shows that 
no point in $X - E$ is the limit of a convergent sequence, so the set $C$ of all limits of convergent sequences from $W$ 
is a subset of $E$.  But $C$ is homeomorphic to $\omega_1$, so cannot be included in a co-stationary $E$.

There is, however, a satisfactory improvement of Theorem \ref{thm:paracompactcountablytight}:

\begin{thm}\label{thm38}
There is a model of form PFA$(S)[S]$ in which a locally compact, normal, countably tight space is paracompact if and only 
if its separable closed subspaces are Lindel\"of, and it does not include a copy of $\omega_1$.
\end{thm}

This follows from:

\begin{thm}[{~\cite{DT2}}]\label{thmConj2}
PFA$(S)[S]$ implies a countably tight, perfect pre-image of $\omega_1$ includes a copy of $\omega_1$.
\end{thm}

The proof of Theorem \ref{thm38} is essentially the same as the proof in \cite{T} of our Theorem 2.2.

Countably tight, hereditarily normal perfect pre-images of $\omega_1$ are rather special:

\begin{defn}
Suppose $\pi : X \to \omega_1$.  We say $Y \subseteq X$ is \textbf{unbounded} if $\pi(Y)$ is unbounded.
\end{defn}

\begin{thm}\label{thm39}
PFA$(S)[S]$ implies that a countably tight, hereditarily normal, perfect pre-image of $\omega_1$ is 
the union of a paracompact space with a finite number of disjoint unbounded copies of $\omega_1$.
\end{thm}
\begin{proof}
By \ref{thmConj2}, the perfect pre-image $X$ includes a copy, $W_1$, of $\omega_1$.  If $W_1$ were bounded, then for some 
$\alpha$, $W_1 \subseteq \pi^{-1}([0, \alpha])$.  But $\pi^{-1}([0, \alpha])$ is compact, and $W_1$ -- being a countably 
compact subspace of a countably tight space -- is closed in $X$ and hence in $\pi^{-1}([0, \alpha])$.  But then $W_1$ is 
compact, contradiction.  Since perfect pre-images of locally compact spaces are locally compact, $X$ is locally compact.  
Since $W_1$ is closed, $X - W_1$ is open and so is also locally compact.  If it is paracompact, we are done; if not, apply 
\ref{thmConj2} to get a copy $W_2$ of $\omega_1$ included in $X - W_1$.  Continue.  The process must end at some finite 
stage, since:

\begin{lem}[{~\cite[3.6]{Ny2}}]
Let $X$ be a $T_5$ space, $\pi : X \to \omega_1$ continuous, $\pi^{-1}(\{\alpha\})$ countably compact for all $\alpha \in 
S$, a stationary subset of $\omega_1$.  
Then $X$ cannot include an infinite disjoint family of closed,
countably compact  
subspaces each with unbounded range.
\end{lem}

Note that the paracompact subspace is the topological sum of $\leq \aleph_1$ $\sigma$-compact subspaces.
\end{proof}

An early version of \cite{DT1} used the axioms $\bsigma^-$ (defined in Section 5), $\PPI$, and the $\aleph_1$-collectionwise Hausdorffness of first countable normal spaces, as well as \ref{thm39} to obtain ``countably compact, hereditarily 
normal manifolds of dimension $> 1$ are metrizable" without the $\mathbf{P}_{22}$ axiom used in \cite{DT1} to get the stronger assertion in 
which ``countably compact" is omitted.

Both of the conditions for paracompactness in \ref{thm38} are necessary:

\begin{eg}
	$\omega_1$ is locally compact, normal, first countable, its separable subspaces are countable, but it is not paracompact.
\end{eg}

\begin{eg}
	Van Douwen's ``honest example'' \cite{vD} is locally compact, normal, first countable, separable, does not include a perfect pre-image of $\omega_1$ (because it has a $G_\delta$-diagonal), but is not paracompact.
\end{eg}

\section{Strengthenings of {PFA}$(S)[S]$}

In addition to ``front-loading'' a PFA$(S)[S]$ model in order to get full collectionwise Hausdorffness, it has also been useful to employ strengthenings of PFA$(S)$ so as to obtain more reflection. E.g. in \cite{LT2} and \cite{T}, \textbf{\textup{Axiom R}} is employed.

\begin{defn}
	$C\subseteq[X]^{<\kappa}$ is \textbf{tight} if whenever $\{C_\alpha:\alpha<\delta\}$ is an increasing sequence from $C$ and $\omega<\text{\normalfont cf}(\delta)<\kappa$, $\bigcup\{C_\alpha:\alpha<\beta\}\in C$.
\end{defn}

\vspace{.2cm}

\noindent
\begin{minipage}[t]{2cm}
\textbf{Axiom R}
\end{minipage}
\begin{minipage}[t]{11.5cm}
If $\mathcal{S}\subseteq[X]^{<\omega_1}$ is stationary and $C\subseteq[X]^{<\omega_2}$ is tight and unbounded, then there is a $Y\in C$ such that $\mathcal{P}(Y)\cap\mathcal{S}$ is stationary in $[Y] ^{<\omega_1}$.
\end{minipage}

\vspace{.5cm}

\textbf{\textup{Axiom R}} (due to Fleissner \cite{Fle})  was obtained by using what is called \textit{PFA$^{++}(S)$} in \cite{LT2}, before forcing with $S$ \cite{LT2}. PFA$^{++}(S)$ holds if PFA$(S)$ is forced in the usual Laver-diamond way. Here we shall use a conceptually simple principle, MM$(S)$, which is forced in a more complicated way, but does not require a larger cardinal. The axiom \textit{Martin's Maximum} was introduced in \cite{FMS}. 

\begin{defn}
	Let $\mathcal{P}$ be a partial order such that forcing with $\mathcal{P}$ preserves stationary subsets of $\omega_1$. Let $\mathcal{D}$ be a collection of $\aleph_1$ dense subsets of $\mathcal{P}$. \textbf{\textup{Martin's Maximum}} \textup{(\textbf{MM})} asserts that for each such $\mathcal{D}$, there is a $\mathcal{D}$-generic filter included in $\mathcal{P}$.
\end{defn}

\begin{thm}[ {\cite{FMS}}]
	Assume there is a supercompact cardinal. Then there is a revised countable support iteration establishing {MM}.
\end{thm}

MM$(S)$ is defined analogously to PFA$(S)$; Miyamoto \cite{M} proved that there is a ``nice'' iteration establishing MM$(S)$ but preserving $S$. One can then define MM$(S)[S]$ analogously to PFA$(S)[S]$.

In order to obtain a model of PFA$(S)[S]$ in which Theorem \ref{thm:paracompactcopyallmodels} holds, we need to improve the model of \cite{LT2} so as to not only have \textbf{Axiom R} but also: 
\newline
\newline
\noindent
\begin{minipage}[t]{1.9cm}
	\textbf{LCN($\aleph_1$)}
\end{minipage}
\begin{minipage}[t]{10.6cm}
	Every locally compact normal space is $\aleph_1$-collectionwise Hausdorff.
\end{minipage}\newline

We shall prove that MM$(S)$ implies:

\noindent\newline
\begin{minipage}[t]{1.6cm}
	\textbf{NSSAT}
\end{minipage}
\begin{minipage}[t]{10.9cm}
	NS$_{\omega_1}$ (the non-stationary ideal on $\omega_1$) is $\aleph_2$-saturated.
\end{minipage}\newline

\noindent
\begin{minipage}{1cm}
	\textbf{SCC}
\end{minipage}
\begin{minipage}[t]{11.5cm}
	Strong Chang Conjecture. Let $\lambda>2^{\aleph_2}$ be a
        regular cardinal. Let $H(\lambda)$ be the collection of
        hereditarily $<\lambda$ sets. Let $M^*$ be an expansion of
        $\langle H_\lambda,\in\rangle$. Let $N\prec M^*$ (i.\,e. $N$
        is an elementary submodel of $M^*$) be countable. Then there
        is an $N'$ such that $N\prec N'\prec M^*$, $N'\cap
        \omega_1=N\cap\omega_1$, and $|N\cap\omega_2|=\aleph_1$. 
\end{minipage}
\newline
\newline

We also note:

\begin{lem}[{ \cite{To}}]
	\textup{MM}$(S)$ implies $2^{\aleph_1}=\aleph_2$.
\end{lem}

With these, we can modify the proof in \cite{LT1} that forcing with a Souslin tree makes \textit{first countable normal spaces $\aleph_1$-collectionwise Hausdorff} to obtain \textit{locally compact normal spaces are $\aleph_1$-collectionwise Hausdorff}, and then, if we wish, front-load the model as in \cite{LT1} to obtain full collectionwise Hausdorffness, using the character reduction method of \cite{W}. More precisely, the crucial new step is:

\begin{thm}\label{thm41}
	Suppose there is a model in which there is a Souslin tree $S$ and in which \textbf{\textup{NSSAT}}, \textbf{\textup{SCC}}, and $2^{\aleph_1}=\aleph_2$ hold. Then $S$ forces that locally compact normal spaces are $\aleph_1$-collectionwise Hausdorff.
\end{thm}

It will be convenient to consider the following intermediate proposition, which implies the three things that we want:
\newline
\newline
\noindent
\begin{minipage}{1cm}
	\textbf{SRP}
\end{minipage}
\begin{minipage}[t]{12.5cm}
			Strong Reflection Principle
                        \cite{To2}. Suppose $\lambda\geq\aleph_2$ and
                        $\mathfrak{Z}\subseteq\mathcal{P}_{\omega_1}(\lambda)$
                        and that for each stationary
                        $T\subseteq\omega_1$, 
			\begin{equation*}
				\{\sigma\in\mathfrak{Z}:\sigma\cap\omega_1\in T\}
			\end{equation*}
			is stationary in $\mathcal{P}_{\omega_1}(\lambda)$. Then for all $X\subseteq\lambda$ of cardinality $\aleph_1$, there exists $Y\subseteq\lambda$ such that:
			\begin{itemize}
				\item[(a)] $X\subseteq Y$ and $|Y|=\aleph_1$;
				\item[(b)] $\mathfrak{Z}\cap\mathcal{P}_{\omega_1}(Y)$ contains a set which is closed unbounded in $\mathcal{P}_{\omega_1}(Y)$.
			\end{itemize}
\end{minipage}
\newline
\newline

 With regard to \textbf{SCC}, Shelah \cite[XII.2.2, XII.2.5]{S} proves that:

\begin{lem}
	If there is a semi-proper forcing P changing the cofinality of $\aleph_2$ to $\aleph_0$, then \textbf{\textup{SCC}} holds.
\end{lem}

There are various versions of \textit{Namba forcing}, e.g. two in \cite{S} and one in \cite{Lar}. All of these change the cofinality of $\aleph_2$ to $\aleph_0$. Larson states in \cite[p.142]{Lar} that his version of Namba forcing preserves stationary subsets of $\omega_1$. In \cite{FMS}, it is shown that a principle, \textbf{SR}, implies \textit{any forcing that preserves stationary subsets of $\omega_1$ is semi-proper}. \textbf{SR} is a consequence of MM \cite{FMS}.
\textup{\textbf{SRP}} is stronger than \textbf{SR} and so:
\newline

\begin{lem}
	\textup{\textbf{SRP}} implies \textup{\textbf{SCC}}.
\end{lem}

\begin{lem}[ \cite{M}]
	\textup{MM}$(S)$ implies \textup{\textbf{SRP}}.
\end{lem}

\begin{lem}[{ \cite[quoted in 45, p.40]{To2}}] 
	\textbf{\textup{SRP}} implies \textbf{\textup{NSSAT}} and $2^{\aleph_1}\leq\aleph_2$.
\end{lem}

For the proof of \ref{thm:paracompactcopyallmodels} we should also remark that:

\begin{lem}
	\textbf{\textup{SRP}} implies \textbf{\textup{Axiom R}}.
\end{lem}

\begin{proof}
	We use an equivalent formulation of \textbf{SRP} due to Feng and Jech \cite{FJ}.
	\vspace{.cm}
	
	\noindent
	\begin{minipage}[t]{1cm}
		\textbf{SRP}
	\end{minipage}
	\begin{minipage}[t]{12.5cm}
		For every cardinal $\kappa$ and every $S\subseteq[\kappa]^\omega$, for every regular $\theta>\kappa$, there is a continuous elementary chain $\{N_\alpha:\alpha\in\omega_1\}$ (with $N_0$ containing some given element of $H(\theta)$, e.g. $S$) such that for all $\alpha$, $N_\alpha\cap\kappa\in S$ if and only if there is a countable $M\prec H(\theta)$ such that $N_\alpha\subseteq M$, $M\cap\omega_1=N_\alpha\cap\omega_1$, and $M\cap\kappa\in S$.
	\end{minipage}\newline
	
	\vspace{.3cm}

			Let $\mathcal{S}$ and $\mathcal{C}$ be as in \textbf{Axiom R}. Choose $\theta$ sufficiently large so that $\mathcal{S},\mathcal{C}\in H(\theta)$ and so that $\theta^{\aleph_1}=\theta$. Let $\{\mathcal{S},\mathcal{C}\}\in N_0$ and let $\{N_\alpha:\alpha\in\omega_1\}$ be as in \textbf{SRP}. By induction on $\alpha\in\omega_1$, choose $Y_\alpha\in C\cap N_{\alpha+1}$ so that $\bigcup (\mathcal{C}\cap N_\alpha)\subseteq Y_\alpha$. Then $\{Y_\alpha:\alpha\in\omega_1\}$ is an increasing chain in $\mathcal{C}$. Therefore $Y=\bigcup_{\alpha\in\omega_1}(N_\alpha\cap\kappa)$ is in $\mathcal{C}$.
			
			$\mathcal{S}^+=\{M\prec H(\theta):M\cap \kappa\in\mathcal{S}\}$ is a stationary subset of $[H(\theta)]^\omega$. This is proved in the same way as 1) of Claim 1.12 on page 196 of \cite{S}. Since $\{N_\alpha:\alpha\in\omega_1\}$ is an element of $H(\theta)$, there is an $M\in\mathcal{S}^+$ such that $\{N_\alpha:\alpha\in\omega_1\}\in M$. Let ${M\cap\omega_1=\delta}$. Obviously $M\cap\kappa\in\mathcal{S}$, and, by continuity, $N_\delta\subseteq M$ and $M\cap\omega_1= N_\delta\cap\omega_1$. It then follows from \textbf{SRP} that $N_\delta\in\mathcal{S}$.
			
			This actually proves that $\{\alpha\in\omega_1:N_\alpha\cap\kappa\in\mathcal{S}\}$ is a stationary subset of $\omega_1$, because we could have put any cub of $\omega_1$ as an element of $M$. Now assume that $\mathfrak{Z}\subseteq[Y]^\omega$ is a cub of $[Y]^\omega$. Choose a strictly increasing $g:\omega_1\to\omega_1$ such that for each $\alpha$, there is a $Z_\alpha\in \mathfrak{Z}$ such that $N_\alpha\cap \kappa\subseteq Z_\alpha\subseteq N_{g(\alpha)}$. If limit $\delta$ satisfies that $g(\alpha)<\delta$ for all $\alpha<\delta$, then we have that $N_\delta\cap\kappa\in\mathfrak{Z}$. This finishes the proof that $\mathcal{S}\cap[Y]^\omega$ is stationary.
\end{proof}


\begin{theorem}
	Suppose we have a model with a Souslin tree $S$ in which \textbf{Axiom R} holds. Then, after forcing with $S$, \textbf{\text{Axiom R}} still holds.
\end{theorem}

\begin{proof}
This is an improvement over \cite{LT2}, which required a stronger axiom, Axiom R$^{++}$, holding in the model. We will use \emph{t.u.b.} as an abbreviation for \emph{tight unbounded}. We must consider two  $S$-names:  $\dot{\mathcal{C}}$ and $\dot{\mathcal{X}}$ where
 $\dot{ \mathcal{C}}$ is forced to be a t.u.b. subset of $[\kappa]^{\omega_1}$ and
$\dot{ \mathcal{X}}$ is forced to be a stationary subset of $[\kappa]^\omega$.
Let us assume that some $s_0\in S$ forces there is no $Y$ in $\dot{ \mathcal{C}}$
such that $\dot{ \mathcal{X}}\cap [Y]^\omega$ is stationary. (It would make the
discussion below easier if we just assumed that $s_0$ was the root of
$S$ -- which one can certainly immediately do if $S$ is a coherent
Souslin tree.)

We first show that $\dot C$ contains a t.u.b.  $C$ from the ground
model. Simply put $Y\in \mathcal{C}$ if every $s\in S$ forces that $Y\in \dot{ \mathcal{C}}$. It
is clear that $\mathcal{C}$ is closed under increasing $\omega_1$-chains. Thus
we just have to show that it is unbounded. Let us enumerate $S$ as
 $\{ s_\alpha : \alpha\in \omega_1\}$. Fix any $Y_0\in
[\kappa]^{\omega_1}$. By recursion choose an increasing chain
$\{Y_\alpha : \alpha\in \omega_1\}$ so that for each $\alpha$,
$\bigcup \{Y_\beta  : \beta < \alpha\}\subseteq Y_\alpha$ and 
 there is an extension $s_\beta$ of $s_\alpha $
 forcing that $Y_{\alpha+1}\in  \dot{ \mathcal{C}}$. This we may do, since
  $s_\alpha$ forces that $\dot{ \mathcal{C}}$ is unbounded. Now let $Y$ be the
 union of the chain $\{ Y_\alpha : \alpha\in \omega_1\}$. Note that
 for each $s\in S$ and each $\beta\in \omega_1$, there is an
 $\beta<\alpha$ such that $s_\alpha$ is an extension of $s$.
 It follows that $s$ forces that $\dot{ \mathcal{C}}\,\cap \{ Y_\alpha : \alpha \in
 \omega_1\}$ is uncountable, hence $s\Vdash Y\in \dot{ \mathcal{C}}$.

 Now we let $\mathcal{X}$ be the set of $x\in [\kappa]^\omega$ such that there
 is some $s\in S$ extending $s_0$
with $s\Vdash x\in \dot{ \mathcal{X}}$. It is clear that $\mathcal{X}$ is
 a stationary subset of $[\kappa]^\omega$ because $s_0$ forces that $\mathcal{X}$
 meets every cub. Now apply \textbf{Axiom R} to choose $Y\in \mathcal{C}$ so that
 $\mathcal{X}\cap [Y]^\omega$ is  a stationary subset of $Y$.

 Now we obtain a contradiction (and thus a proof)
 by showing that
  there is an extension  $s\in S$ of $s_0$ that forces that
$\dot{ \mathcal{X}}\cap [Y]^\omega$ is stationary.
 Let $\{ y_\alpha :\alpha\in \omega_1\}$ be an enumeration
 of $Y$. Let $\mathcal{E}$ be the set of $\delta\in \omega_1$ such that
 $x_\delta = \{y_\alpha :\alpha\in \delta\}\in \mathcal{X}$. Notice that
 $\{ \{ y_\alpha : \alpha \in \delta \} : \delta\in \omega_1\}$ is
 a cub in $[Y]^\omega$. Thus it follows that $\mathcal{E}$ is stationary.
 In fact, if $\mathcal{E}'$ is any stationary subset of  $\mathcal{E}$,
 then $\mathcal{E}'$ is also a stationary subset of $[Y]^\omega$.
 
 For each $\delta\in \mathcal{E}$ choose $s_\delta\in S$ above $s_0$
so that
 $s_\delta\Vdash x_\delta\in \dot{ \mathcal{X}}$ (as per the definition of
 $ \mathcal{X}$).  Now we have a name $\dot{ \mathcal{E}} = \{ (x_\delta, s_\delta) :
\delta\in \omega_1\}$. We prove that there is some $s\in S$
above $s_0$ 
that
 forces that $\dot{ \mathcal{E}}$ is stationary. Thus such an $s$ forces
 that $\dot{ \mathcal{X}}\cap [Y]^\omega$ is stationary as required.

Let $s_0$ be on level $\alpha_0$ of $S$.
 There is a  $\gamma>\alpha_0$ so that
each member of $S_\gamma$
decides if $\dot{ \mathcal{E}}$ is stationary. 
 Also, for each $\bar s\in S_\gamma$ that forces
 $\dot{ \mathcal{E}}$ is not stationary, there is a cub $\mathcal{C}_{\bar
   s}$ of $\omega_1$ that $\bar s$ forces is disjoint
 from $\dot{ \mathcal{E}}$. Choose any $\delta $ in the intersection of those
 countably many cubs that is also in $\mathcal{E}$. Clearly if
 $\bar s\in S_\gamma$ is compatible with $s_\delta$, then
 $\mathcal{C}_{\bar s}$ did not exist since $\bar s \cap s_\delta$
 would   force that $\delta \in
 \mathcal{C}_{\bar s}\cap  \dot{ \mathcal{E}}$. This completes the proof, since that element $\bar s$
 is above $s_0$ and forces that $\dot{ \mathcal{X}}\cap [Y]^\omega$ is stationary.
\end{proof}


\begin{cor}
	\textup{MM}$(S)[S]$ implies \textbf{\textup{Axiom R}}.
\end{cor}

We next need:

\begin{lemma} [P. Larson]\label{larson}
	Suppose
	\begin{itemize}
		\item[(1)]\textbf{\textup{NSSAT}}, and
		\item[(2)] for sufficiently large $\theta$ and stationary $E\subseteq\omega_1$, for any $X\in H(\theta)$, there is a Chang model $M$ with $M\cap\omega_1\in E, X\in M$ and $|M\cap\omega_2|=\aleph_1$.
	\end{itemize}
	Then if $\{A_\alpha:\alpha<\omega_2\}$ are stationary subsets of $\omega_1$, $M\cap\omega_1=\delta$ is in uncountably many $A_\alpha,\alpha\in M$.
\end{lemma}

\begin{proof}
	It is well known that $\mr{NS}_{\omega_1}$ is $\aleph_1$-complete, since the diagonal union of $\aleph_1$ non-stationary subsets of $\omega_1$ is non-stationary. It follows that $\mathcal{P}(\omega_1)/\mr{NS}_{\omega_1}$ is a complete Boolean algebra, because (1) says it satisfies the $\aleph_2$-chain condition. Since it is complete, for each $\alpha<\omega_2$ there is a stationary $B_\alpha$ which is the sup of $\{A_\beta:\beta\in(\alpha,\omega_2)\}$. Let $E$ be the inf of the family of $B_\alpha$'s. By saturation, $E$ is really the inf of an $\aleph_1$-sized family, and so is itself stationary. Given any $\alpha\in\omega_2$, we can find an $\eta(\alpha)>\alpha$ such that the diagonal union of $\{A_\beta:\beta\in(\alpha,\eta(\alpha))\}$ includes $E$, mod $\mr{NS}_{\omega_1}$. It follows that there is a cub $C\subseteq \omega_2$ such that for each $\alpha\in C$, there is a subset of $\{A_\beta:\beta\in(\alpha,\alpha^+)\}$ of cardinality $\aleph_1$ with diagonal union including $E$, mod $\mr{NS}_{\omega_1}$, where $\alpha^+$ denotes the next element of $C$ after $\alpha$.
	
	Now let $M$ be an elementary submodel of a suitable $H(\theta)$, with $\langle A_\alpha:\alpha<\omega_2\rangle$, $E$, and $C\in M$ and $\delta=M\cap\omega_1\in E$, $|M\cap\omega_2|=\aleph_1$. We claim $\delta$ is an element of uncountably many $A_\alpha,\alpha\in M$.
	
	Since the cub $C$ divides $\omega_2$ into $\aleph_2$ disjoint intervals, $C\hspace{.03cm}\cap M$ divides $\omega_2\,\cap M$ into $\aleph_1$ disjoint intervals. Choose any one of these intervals $J$. There is a family $\mathcal{F}_J=\{F_\gamma:\gamma<\omega_1\}$ in $M$ consisting of $A_\alpha$'s indexed in the interval $J$, with diagonal union including $E$, mod $\mr{NS}_{\omega_1}$. Then there is a cub $D_J$ in $M$ disjoint from $E\setminus\nabla\mathcal{F}_J$. $D_J\cap M$ is unbounded in $M$, so $\delta=M\cap\omega_1\in D_J$, so $\delta\notin E\setminus\nabla\mathcal{F}_J$. Then $\delta\in \nabla\mathcal{F}_J$ so $\delta\in F_\gamma$ for some $\gamma\in M\cap\omega_1$ and therefore $\delta$ is in some $A_\xi$ with $\xi\in J$.
\end{proof}

We shall finish the proof that MM$(S)[S]$ implies \textbf{LCN}$(\aleph_1)$ in Section 4, but first let us note another advantage of stating MM$(S)[S]$ as a hypothesis is that we can often avoid front-loading to get collectionwise Hausdorffness, since \textbf{Axiom R} provides enough reflection. For example,

\begin{thm}\label{312}
	\textup{MM}$(S)[S]$ implies a locally compact, hereditarily normal space is hereditarily paracompact if and only if it does not include a copy of $\omega_1$.
\end{thm}

\begin{proof}
	As usual, we may assume the space does not include a perfect pre-image of $\omega_1$. The proof for that case in \cite{T} uses \textit{P-ideal Dichotomy}, $\bsigma$, $\aleph_1$-collectionwise Hausdorffness, and \textbf{Axiom R}. We can get all of these from MM$(S)[S]$. (Todorcevic \cite{To} proved that PFA$(S)[S]$ implies P-ideal Dichotomy; a proof was published in \cite{D2}.)
\end{proof}

Similar considerations enable us to prove:

\begin{thm}\label{thm312}
	\textup{MM}$(S)[S]$ implies a locally compact, normal, countably tight space is paracompact if and only if its separable closed subspaces are Lindel\"of, and it does not include a copy of $\omega_1$.
\end{thm}

We thank Paul Larson for Lemma \ref{larson} and several discussions concerning the material in this section. Next, we need to do some topology.

\section{Getting locally compact normal spaces collectionwise Hausdorff}

\begin{lem}
	Let $X$ be a locally compact normal space and suppose $Y$ is a closed discrete subspace of $X$ of size $\aleph_1$. Then there is a locally compact normal space $X'$ with a closed discrete subspace $Y'$ of size $\aleph_1$, such that if $Y'$ is separated in $X'$, then $Y$ is separated in $X$, but each point in $Y'$ has character $\leq\aleph_1$.
\end{lem}

\begin{proof}
	By Watson's character reduction technique \cite{W}, there is a discrete collection of compact subsets of $X$, $\mathcal{K}=\{K_y:y\in Y\}$, such that $y\in K_y$, and each $K_y$ has character $\leq\aleph_1$. Let $X'$ be the quotient of $X$ obtained by collapsing each $K_y$ to a point $y'$. This collapse is a perfect map, so preserves normality and local compactness, and it is clear that $\{y':y'\in Y\}$ is separated if and only if $\{K_y:y\in Y\}$ is separated, and that $Y$ is separated if $\{K_y:y\in Y\}$ is.
\end{proof}

\begin{lem}\label{lem47}
	Suppose $X$ is a locally compact normal space of Lindel\"of degree $\aleph_1$ with an uncountable closed discrete subspace. Then there is a continuous image of $X$ of weight $\aleph_1$ enjoying the same properties.
\end{lem}

\begin{proof}
	Let $\mathcal{U}$ be an open cover of $X$ of size $\aleph_1$ with each member of $\mathcal{U}$ a cozero set with compact closure. Without loss of generality, assume that for each $x\in X$ there is a $U\in\mathcal{U}$ such that $x\in U$ and $U$ meets at most one element of a given closed discrete set $D$ of size $\aleph_1$. Also without loss of generality, assume $\mathcal{U}$ is closed under finite intersections. For each $U\in\mathcal{U}$, let $f_U:X\to[0,1]$ with $U=f^{-1}_U((0,1])$. Define an equivalence relation on $X$ by letting $x_0\!\!\sim\!\! x_1$ if $f_U(x_0)=f_U(x_1)$ for all $U\in\mathcal{U}$. Let $X/\!\!\sim$ be the quotient set, with $\pi:X\to X/\!\!\sim$ the projection. Topologize $X/\!\!\sim$ by taking as base all sets of form $\pi(U)$, $U\in\mathcal{U}$. Then $X/\!\!\sim$ is $\textup{T}_{3 \frac{1}{2}}$ and of weight $\leq\aleph_1$. To see the former, consider $X$ as embedded in $[0,1]^{C^*(X)}$ by $e(x)=(f(x))_{f\in C^*(X)}$. Let $p:[0,1]^{C^*(X)}\to[0,1]^{\{f_U:U\in\mathcal{U}\}}$ be given by $(x_f)_{f\in C^*(X)}\to(x_{f_U})_{U\in\mathcal{U}}$, i.e. $p$ projects onto only those coordinates in $C^*(X)$ which are $f_U$'s. Then $X/\!\!\sim\; =p\circ e(X)$.
	
	The projection map $\pi$ is closed, for let $F\subseteq X$ be closed and suppose $y\in\overline{\pi[F]}$. Claim $y\in\pi[F]$. $y\in\pi[U]$ for some $U\in\mathcal{U}$; note $\pi^{-1}(\pi[U])=U$ for if $\pi(x)\in\pi[U]$, $x\sim x_0$ for some $x_0\in U$. Then $f_V(x)=f_V(x_0)$ for every $V\in\mathcal{U}$. But $U=f_U^{-1}((0,1])$. Thus $f_U(x)=f_U(x_0)\in(0,1]$, which implies $x\in U$. So $\stackrel{ }{\overline{U}=\overline{\pi^{-1}\left(\pi[U]\right)}}$ is compact. Suppose $y\notin\pi[F]$. Then $y\notin\pi[F\cap\overline{U}]$, which is compact. Then $\pi[U]\setminus\pi[F\cap\overline{U}]$ is a neighborhood of $y$ disjoint from $\pi[F]$.
	
	Since $\pi$ is closed and $X$ is normal, $X/\!\!\sim$ is normal. It is clear that $\pi[D]$ is closed discrete. By continuity, $\pi[\overline{U}]\subseteq\overline{\pi[U]}$; $\pi[\overline{U}]$ is a closed set including $\pi[U]$, so including $\stackrel{ }{\overline{\pi[U]}}$, so $\pi[\overline{U}]=\overline{\pi[U]}$, so $X/\!\!\sim$ is covered by open sets with compact closures, so it is locally compact.
\end{proof}

\begin{lem}\label{lem48}
	In any model obtained by forcing with a Souslin tree $S$, any locally compact normal space with a dense Lindel\"of subspace has countable extent.
\end{lem}

\begin{proof}
	Suppose $X_0$ is a locally compact normal space with an uncountable closed discrete subspace, which we may conveniently label as $\omega_1$, and a dense Lindel\"of subspace $L$. Via normality, we can find a closed subspace $X_1$ with $\omega_1$ in its interior which is covered by $\aleph_1$-many open sets with compact closures. Without loss of generality, we may assume $X_1=\overline{\mr{int}\,X_1}$. $L$ is dense in $\mr{int}\,X_1$, so $L\cap(\mr{int}\,X_1)$ is dense in $X_1$. Then $\stackrel{ }{\overline{L\cap \mr{int}\,X_1}\cap X_1}$ is a dense Lindel\"of subspace of $X_1$.
	
	Thus, without loss of generality, we may as well assume our original space $X_0$ has a cover by $\aleph_1$-many open sets, each with compact closure. Without loss of generality, we may assume each is a cozero set and indeed is $\sigma$-compact. By Lemma \ref{lem47}, there is a continuous image of $X_0$ --- call it $X$ --- which is also locally compact, normal, has an uncountable closed discrete subspace, and has weight $\aleph_1$. Since both density and Lindel\"ofness are preserved by continuous functions, $X$ also has a dense Lindel\"of subspace. Thus it suffices to find a contradiction for the special case in which the weight of our space is $\aleph_1$.
	
	 For $\delta\in\omega_1$ and a cub $C\subseteq\omega_1$, let $\delta^+(C)$ denote the minimum element of $C$ greater than $\delta$. Without loss of generality, we may assume our cubs only consist of limit ordinals. For a cub $C$, we use $\mr{Fix}(C)$ to denote the set $\{\delta\in C:\text{order-type}(C\cap\delta)=\delta\}$. Let $S_\delta$ be the $\delta$th level of the Souslin tree.
	
	As usual, we work in the ground model and fix names $\dot{\mathcal{B}} = \{\dot{B}_\alpha:\alpha\in\omega_1\}$ for a base of $X$ consisting of open sets with compact closures. It is convenient to assume that $\{\dot{B}_n:n\in\omega\}$ is forced to have dense union. Again, we let $\omega_1$ label a closed discrete subspace and let $\{\dot{U}(\alpha,\xi):\xi\in\omega_1\}$ be a subset of $\dot{\mathcal{B}}$ forced to be a local base at $\alpha$. Without loss of generality, assume each $B_n$ is disjoint from the closed discrete set $\omega_1$. Fix a cub $C_0$ such that for each $\delta\in C_0$ and each $s\in S_\delta$, $s$ decides all equations of the form $\dot{B}_\alpha\cap\dot{B}_\beta =\emptyset$, for $\alpha,\beta<\delta$. Also assume that for each $s\in S_\delta\;(\delta\in C_0)$ and each $\xi,\beta\in\delta$, there is an $\alpha\in\delta$ such that $s$ forces that $\dot{U}(\xi,\beta)=\dot{B}_\alpha$.
	
	It is convenient to assume that $S$ is $\omega$-branching (specifying any infinite maximal antichain above each element would serve the same purpose). We can use $C_1=\mr{Fix}(C_0)$ to define a partition $\dot{f}$ of $\omega_1$ so that for each $\xi\in\omega_1$ and each $s\in S_{\xi^+(C_1)}$, $s^\smallfrown j$ forces that $\dot{f}(\xi)=j$. Now we choose two (names of) functions $\dot{h}_1$ and $\dot{h}_2$ witnessing normality as follows:
	\begin{itemize}
		\item[(1)] For each $j\in\omega$ and each $i\in 2$, let $\dot{W}^i_j=\bigcup\{\dot{U}(\xi,h_i(\xi)):\xi\in \dot{f}^{-1}(j)\}$,
		\item[(2)] the $\{\dot{W}^1_j:j\in\omega\}$ form a discrete family,
		\item[(3)] the closure of $\dot{W}^2_j$ is included in $\dot{W}^1_j$.
	\end{itemize}
	
	Choose any countable elementary submodel $M$ with all the above as members of $M$, such that $\delta = M\cap\omega_1$ is an element of $C_1$. We know that there is a name of an integer $\dot{J}_\delta$ satisfying that it is forced that $\dot{U}(\delta,0)\cap\dot{W}_j$ is empty for all $j\geq\dot{J}_\delta$. Choose any $s\in S$ of height at least $\delta^+(C_1)$ that decides a value $J$ for $\dot{J}_\delta$. Let $\bar{s} = s\upharpoonright\delta^+(C_1)$. Notice that $\bar{s}$ decides the truth value of the equation ``$\dot{U}(\delta,0)\cap\dot{B}_\alpha=\emptyset$'', for all $\alpha\in M$. For each $n,j\in\omega$, $s$ and hence $s\upharpoonright\delta$ forces that the closure of $\dot{W}^2_j\cap \dot{B}_n$ is included in $\dot{W}^1_j$. By elementarity and compactness, this implies there is a finite $\dot{F}_{j,n}\subseteq\delta$ such that $s\upharpoonright\delta$ forces that $\dot{W}^2_j\cap \dot{B}_n\subseteq\bigcup\{\dot{B}_\eta:\eta\in \dot{F}_{j,n}\}\subseteq\dot{W}^1_j$. But now $\bar{s}$ forces $\dot{U}(\delta,0)\cap(\bigcup\{\dot{B}_\eta:\eta\in \dot{F}_{j,n}\})$ is empty for all $n$ and all $j\geq J$.
	
	On the other hand, fix any $j\geq J$ and consider what
        $\bar{s}^\smallfrown j$ is forcing. This forces that
        $\dot{f}(\delta)=j$ and that $\delta\in W^2_j$, and so
        $\delta$ is in the closure of the union of the sequence
        $\{\dot{U}(\delta,0)\cap(\bigcup\{\dot{B}_\eta:\eta\in
        F_{j,n}\}):n\in\omega\}$. This is a contradiction. 
\end{proof}

\begin{cor}\label{cor410}
	In any model obtained by forcing with a Souslin tree, if $X$ is locally compact normal, $D$ is a closed discrete subspace of $X$ of size $\aleph_1$ and $\{U_\alpha:\alpha\in\omega_1\}$ are open sets with compact closures, then for any countable $T\subseteq\omega_1$, $\stackrel{}{\overline{\bigcup\{U_\alpha:\alpha\in T\}}}\cap\; D$ is countable.
\end{cor}

\begin{proof}
	${\bigcup\{{\overline{U}_\alpha}:\alpha\in T\}}$ is dense in $\overline{\bigcup\{U_\alpha:\alpha\in T\}}$, which is locally compact normal.
\end{proof}

Getting back to the proof of \ref{thm41}, let us assume we are in a
model of $\mr{MM}(S)$ and that we have an $S$-name $\dot{X}$ for a
locally compact normal space, with a closed discrete subspace labeled
as $\omega_1$, with each of its points having character
$\aleph_1$. Let us note that it follows from character reduction and Lemma \ref{LT1} 
that if there is a discrete expansion of $\omega_1$ into compact
$G_\delta$'s, then $\omega_1$ will have a separation. In fact,
 even more, it is shown in \cite[Theorem 12]{T3} that if $\omega_1$
is forced to have an expansion by compact $G_\delta$'s that is 
 $\sigma$-discrete, then $\omega_1$ will be separated. Since our proof
 is by contradiction, we will henceforth 
assume that it is forced (by the root of $S$)
that there is no expansion of $\omega_1$ into a $\sigma$-discrete
family of compact $G_\delta$'s.
 
For each $\xi,\alpha\in\omega_1$, let
$\dot{U}(\xi,\alpha)$ be the name of the $\alpha$th neighbourhood from
a local base for $\xi$ with $\dot U(\xi,0)$ forced to have compact
closure.  Corollary \ref{cor410}, and the fact that $S$ is ccc,
 ensure that for each $\delta\in
\omega_1$, every element of $S$ forces that 
$\omega_1 \cap \overline{\bigcup\{\dot{U}(\xi,0):\xi<\delta\}}$ 
is bounded by $\gamma$ for some $\gamma\in \omega_1$.
Therefore there is a  cub $C_0$ such that 
without loss of generality, we can assume that each of the
following is forced by each element of $S$:
\begin{enumerate}
\item  for each $\delta\in C_0$, $\omega_1\cap
\stackrel{}{\overline{\bigcup\{\dot{U}(\xi,0):\xi<\delta\}}}$ 
is included in  $\delta^+(C_0)$,
\item for all $\beta\neq \xi$ in $\omega_1$, $\beta\notin \dot{U}(\xi,0)$,
\item  for all $\xi,\alpha\in \omega_1$ 
  $\dot{U}(\xi,\alpha)\subseteq \dot{U}(\xi,0)$ and has compact closure,
\item  for each limit $\delta\in \omega_1$, the sequence
  $\{\dot{U}(\xi,\alpha):\alpha<\delta\}$ is a 
\emph{regular filter}, i.e. each finite intersection of these includes
the closure of another. 
\end{enumerate}

For an $S$-name $\dot{h}$ of a function from $\omega_1$ to $\omega_1$,
let $\dot{U}(\xi,\dot{h})$ stand for $\dot{U}(\xi,\dot{h}(\xi))$. For
limit $\delta$, let $\dot{Z}(\xi,\delta)$ denote the $S$-name of the
compact $G_\delta$ equal to
$\bigcap\{\dot{U}(\xi,\alpha):\alpha<\delta\}$. For a cub $C$ and
ordinal $\xi$, we also use $\dot{Z}(\xi,C)$ as an abbreviation for
$\dot{Z}(\xi,\xi^+(C))$. 

Fix an enumeration $\{C_\gamma:\gamma\in\omega_2\}$ 
for a base for the
cubs on $\omega_1$ (each containing only limit ordinals), chosen so
that $C_0$  is as above and for $0<\lambda \in \omega_2$,
$C_\lambda \subseteq \mr{Fix}(C_0)$ and
 $C_\lambda\setminus\mr{Fix}(C_\gamma)$ is countable for all
$0\leq\gamma<\lambda$. We can do this by taking diagonal intersections,
since \textbf{SRP} implies $2^{\aleph_1}=\aleph_2$.

For each $\delta\in C_0$, let $\beta(\delta) = \delta^+(C_0)$.
 Since $\dot{Z}(\xi,C_\gamma)\subseteq
\dot{U}(\xi,C_\gamma)$  for all $\xi\in \omega_1$
for all  $\delta\in C_\gamma$, 
$\beta(\delta)<\delta^+(C_\gamma)$, and so
 it is forced that:
\begin{equation*}
\overline{\bigcup\{\dot{Z}(\xi,C):\xi<\delta\}}
\cap\omega_1\subseteq\beta(\delta).   
\end{equation*}

We can also assume that for all cubs $C\subseteq C_0$, there is an
 $S$-name $\dot A$, that is forced to be a
stationary subset of 
$\mr{Fix}(C)$ satisfying:
\begin{equation*}
	(\forall  s\in S)(\forall \delta)
~~ s\Vdash \left(\delta\in \dot A \ \Rightarrow
 (\exists\alpha\in[\delta,\beta(\delta)])\;
\alpha\in\overline{\bigcup\{\dot{Z}(\xi,C):\xi<\delta\}}~\right).
\end{equation*}

The reason we can make this assumption is
that we are assuming there is no $\sigma$-discrete expansion of
 $\omega_1$ by compact $G_\delta$'s. If, in the extension, the 
set $A = \{ \delta : 
\overline{\bigcup\{\dot{Z}(\xi,C):\xi<\delta\}} \not\subseteq \delta\}$
were not stationary, then there would be a $\lambda\in \omega_2$
such that $A\cap C_\lambda$ is empty.
Since the cub $C_\lambda$ divides $\omega_1$ into
countable pieces, we see that we can expand the points in $\omega_1$
into a $\sigma$-discrete collection of compact $G_\delta$'s.

For each $\lambda \in \omega_2$, let $\dot A_\lambda$ denote the 
name of the stationary set just described. 
For any $B\subseteq\omega_1$, we will write
\begin{equation*}
	\alpha\in\langle\dot{Z}(\xi,C):\xi<\delta\rangle'
\end{equation*}
to mean that $\alpha$ is a limit point of that sequence of sets.

Fix any function $e:S\to\omega$ with the property that for all
$\delta\in\omega_1$, $e\restriction S_\delta$ is one-to-one. For an ordinal
$\gamma\in\omega_2$, we use $\dot{f}_\gamma$ for the $S$-name of the
function from $\omega_1$ into $\omega$ given by the property that each
$s\in S_{\xi^+(C_\gamma)}$ forces that
$\dot{f}_\gamma(\xi)=e(s)$. Thus $\dot{f}_\gamma$ partitions
$\omega_1$ into a discrete collection of countably many closed
subsets. Then let $\{\dot{W}(\gamma,n):n\in\omega\}$ be a discrete
collection of open sets separating the $\dot{f}_\gamma^{-1}(n)$'s. Fix
$n\in\omega$. By normality, there is an open $\dot{V}_n$ such that $S$
forces
$\dot{f}^{-1}_\gamma(n)\subseteq\dot{V}_n\subseteq\dot{\overline{V}}_n\subseteq\dot{W}(\gamma,n)$. For
each $\xi\in\dot{f}^{-1}_\gamma(n)$, there is an
$\alpha_\xi\in\omega_1$ such that $S$ forces
$\dot{U}(\xi,\alpha_\xi)\subseteq\dot{V}_n$. Let
$\zeta_n(\gamma)\in\omega_2$ be such that for
$\xi\in\dot{f}^{-1}_\gamma(n), \xi<\rho\in C_{\zeta_n(\gamma)}$
implies $\alpha_\xi<\rho$. Then $S$ forces $\{\dot{Z}(\xi,
C_{\zeta_n(\gamma)}):\xi\in\dot{f}_\gamma^{-1}(n)\}\subseteq\dot{V}_n$. We
then can find a $C_{\zeta(\gamma)}$ included in each
$C_{\zeta_n(\gamma)}$ such that for every $n\in\omega$, $S$ forces
$\{\dot{Z}(\xi,C_{\zeta(\gamma)}):\zeta\in\dot{f}^{-1}_\gamma(n)\}\subseteq
\dot{V}_n$. Thus 
\begin{equation*}
	\overline{\bigcup\{\dot{Z}(\xi,C_{\zeta(\xi)}):\xi\in f_\gamma^{-1}(n)\}}
\subseteq\dot{W}(\gamma,n).
\end{equation*}
Then we can get a $\zeta(\gamma)$ that works for all $n$.

By recursion on $\gamma\in\omega_2$, we can choose
$\zeta(\gamma)\geq\gamma$ as above, so that the sequence
$\{\zeta(\gamma):\gamma\in\omega_2\}$ is strictly increasing. 
For each
$\gamma$, we have the $S$-name
 $\dot A_{\zeta(\gamma)}$ as above.
It is immediate that $A_\gamma = \{ \delta :
 (\exists s\in S) s\Vdash \delta\in \dot A_{\zeta(\gamma)}\}$ is 
a stationary set. In other words, $\delta \in A_\gamma$ 
implies there is some $s\in S$ and $\eta \in [\delta,\beta(\delta)]$
such that 
 $s\Vdash \eta \in 
\langle\dot{Z}(\xi,
C_{\zeta(\gamma)}):\xi\in\delta\rangle'$.

By \textbf{SCC} and \ref{larson} we may assume there is an elementary
submodel $M$ of some $\langle H(\theta),\{\langle
\gamma,\zeta(\gamma), A_\gamma\rangle:\gamma\in\omega_2 \}\rangle$,
with $M\cap\omega_1=\delta<\omega_1$, $|M\cap\omega_2|=\aleph_1$, and
an uncountable $\{\gamma_\alpha:\alpha\in\omega_1\}\subseteq
M\cap\omega_2$, so that $\delta\in A_{\gamma_\alpha}$ for all
$\alpha\in\omega_1$. 

For each $\alpha\in\omega_1$ choose $s_{\alpha}\in S$,
 $\eta_\alpha\in[\delta,\beta(\delta)]$ 
such that $s_\alpha\Vdash
\eta_\alpha\in\langle\dot{Z}(\xi,C_{\zeta(\gamma_\alpha)}):
\xi\in\delta\rangle'$. We 
may assume $s_\alpha$ is on a level at least as high as
$\delta^+(C_{\gamma_\alpha})$. We may also assume that if
$\alpha<\beta\in\omega_1$, then $\gamma_\alpha<\gamma_\beta$. We may
also assume that the height of $s_\alpha$ is less than the height of
$s_\beta$, for $\alpha<\beta$, so that
$\{s_\alpha:\alpha\in\omega_1\}$ is an uncountable subset of
$S$. Therefore  there is an $\eta\in [\delta,\beta(\delta)]$
such that $L = \{ \alpha : \eta_\alpha = \eta\}$ is uncountable.
Also, as is well-known for Souslin trees, 
there is an $\bar{s}\in S$, such that
$\{s_\alpha :\alpha\in L\}$ 
includes a dense subset of $\{s\in S:\bar{s}<s\}$. By passing to an
uncountable subset, we may assume that $\bar{s} <s_\alpha$ for all
$\alpha\in L$ and that $\bar{s}$ is on a level
above $\delta$. Similarly  we may assume that
 for all $\xi,\rho<\delta$, $\bar{s}$ has
decided  the statement
\begin{equation*}
	\dot{U}(\eta,0)\cap\dot{Z}(\xi,\rho)\neq\emptyset\quad
\text{ for all }\xi,\rho<\delta.
\end{equation*}
Now choose any $\alpha\in L$ (e.\;g.\;the least one), and then
choose an infinite sequence $\{\beta_l:l\in\omega\}\subseteq
L\setminus(\alpha+1)$ so that
$s_{\beta_l}\upharpoonright \delta^+({C_{\gamma_\alpha}})$
 are all distinct. For each
$l$, let $e(s_{\beta_l}\upharpoonright\delta^+({C_{\gamma_\alpha}})~)=n_l$. 
\newline
\newline
\noindent\textbf{Main Claim:} $\quad\bar{s}\Vdash
 (\forall
 l\in\omega)\left(\dot{W}(\gamma_\alpha,n_l)\cap\dot{U}(\eta,0)\neq
   0\right).$ 
\newline
\newline
Once this claim is proven we are done, because we then have that
$\bar{s}$ forces that $\dot{U}(\eta,0)$ cannot have compact closure,
because it meets infinitely many members of the discrete family
$\{\dot{W}(\gamma_\alpha,n):n\in\omega\}$. 

To prove the claim, first note that there is a tail of
$C_\zeta(\gamma_{\beta_l})\cap\delta$ included in
$C_{\zeta(\gamma_{\alpha})}$. To see this, recall
$C_{\zeta(\gamma_\alpha)}\setminus\mr{Fix}(C_\zeta(\gamma_{\beta_l}))$
is countable, so some tail of $\mr{Fix}(C_\zeta(\gamma_{\beta_l}))$ is
included in $C_{\zeta(\gamma_\alpha)}$. By elementarity, since
$\gamma_\alpha$ and $\gamma_\beta$ are in $M$, a tail of
$\mr{Fix}(C_\zeta(\gamma_{\beta_l}))\cap M$ is included in
$C_{\zeta(\gamma_\alpha)}\cap M$, so a tail of
$\mr{Fix}(C_\zeta(\gamma_{\beta_l}))\cap\delta$ is included in
$C_{\zeta(\gamma_\alpha)}$.

Since there is a tail of $C_{\zeta(\gamma_{\beta_l})}\cap \delta$
included in $C_{\zeta(\gamma_\alpha)}$, $\dot{Z}(\xi,
C_{\zeta(\gamma_{\beta_l})})\subseteq\dot{Z}(\xi,C_{\zeta(\gamma_\alpha)})$
for each $\xi<\delta$ (at least on a tail --- which is all that
matters for limits above $\delta$). Then $s_{\beta_l}$ forces that
$\eta $ is a limit of the sequence  
\begin{equation*}
	\langle\dot{Z}(\xi, C_{\zeta(\gamma_\alpha)}):
\xi\in\delta\text{ and }\dot{f}_{\gamma_\alpha}(\xi)= n_l\rangle.
\end{equation*}

Of course this means that $s_{\beta_l}$ forces that
$\dot{U}(\eta,0)$ meets $\dot{Z}(\xi, C_{\zeta(\gamma_\alpha)})$ for
cofinally many $\xi<\delta$ such that
$s_{\beta_l}\upharpoonright\gamma_\alpha\Vdash
\dot{f}_{\gamma_\alpha}(\xi)=n_l$. But $\bar{s}$ has already decided
the value of $\dot{f}_{\gamma_\alpha}\upharpoonright\delta$, and
$\bar{s}$ already forces $\dot{U}(\eta,0)\cap\dot{Z}(\xi,
C_{\zeta(\gamma_\alpha)})\neq\emptyset$ whenever $s_{\gamma_\beta}$
does. In particular then, $\bar{s}$ forces there is a $\xi$ with
$\dot{f}_{\gamma_\alpha}(\xi)=n_l$ (and so $\dot{Z}(\xi,
C_{\zeta(\gamma_\alpha)})\subseteq\dot{W}(\gamma_\alpha,n_l)$) and
$\dot{U}(\eta,0)\cap\dot{Z}(\xi,
C_\zeta(\gamma_\alpha))\neq\emptyset$.\qed

For the record, let us state what we have accomplished:

\begin{thm}
	\textup{MM}$(S)[S]$ implies \textbf{LCN}$(\aleph_1)$.
\end{thm}

\begin{cor}
	There is a model of \textup{MM}$(S)[S]$ in which \textbf{LCN} holds,  \emph{i.e. every locally compact normal space is collectionwise Hausdorff.}
\end{cor}

\section{Large Cardinals and the MOP}

In \cite{DT2} we showed that large cardinals are not required to obtain the consistency of every \textit{locally compact perfectly normal space is paracompact}. It is interesting to see which other PFA$(S)[S]$ results can be obtained without large cardinals. The standard method used was pioneered by Todorcevic in \cite{To3} and given several applications in \cite{D}, all in the context of PFA results. In the context of PFA$(S)[S]$, it is referred to in \cite{To} and actually carried out in \cite{DT1} for a version of {P-ideal Dichotomy} and for \textbf{\textup{PPI}}. It is routine to get additionally that such models are of form MA$_{\omega_1}(S)[S]$ by interleaving additional forcing. In \cite{DT2} we pointed out that such methods can give models in which in addition the following holds:

\vspace{.5cm}

\noindent
\begin{minipage}[t]{3.2cm}
\textbf{$\bsigma^{\bm{-}}$(sequential)}
\end{minipage}
\begin{minipage}[t]{10.3cm}
	In a compact sequential space, each locally countable subspace of size $\aleph_1$ is $\sigma$-discrete.
\end{minipage}

\vspace{.5cm}

A modification of such a proof produces a model in which the following proposition (see \cite{FTT}) holds:

\vspace{.5cm}
\noindent
\begin{minipage}[t]{3cm}
\textbf{$\bsigma$(sequential)}
\end{minipage}
\begin{minipage}[t]{10.5cm}
	Let $X$ be a compact sequential space. Let $Y\subseteq X$, $|Y|=\aleph_1$. Suppose $\{W_\alpha\}_{\alpha\in\omega_1}$, $\{V_\alpha\}_{\alpha\in\omega_1}$ are open subsets of $X$ such that:
	\begin{itemize}
		\item[(1)] $W_\alpha\subseteq\overline{W_\alpha}\subseteq V_\alpha,$
		\item[(2)] $|V_\alpha\cap Y|\leq\aleph_0$,
		\item[(3)] $Y\subseteq\bigcup\{W_\alpha:\alpha\in\omega_1\}$.
	\end{itemize}
	Then $Y$ is $\sigma$-closed discrete in $\bigcup\{W_\alpha:\alpha\in\omega_1\}$.
\end{minipage}
\vspace{.5cm}

Without the parenthetical ``sequential'', $\bsigma^-$ and $\bsigma$ refer to the corresponding propositions obtained by replacing ``sequential'' by countably tight'', which follow from their sequential versions if one has

\vspace{.5cm}

\noindent
\begin{minipage}[t]{3.3cm}
\textbf{Moore-Mr\'owka}
\end{minipage}
\begin{minipage}[t]{10.2cm}
	Every compact countably tight space is sequential.
\end{minipage}

\vspace{.5cm}

It follows easily from \textbf{Moore-Mr\'owka} that \emph{locally compact countably tight spaces are sequential}. A proof of \textbf{Moore-Mr\'owka} from PFA$(S)[S]$ is sketched in \cite{To} and the author remarks that, by the usual methods, large cardinals are not necessary. Thus, one can obtain a model of MA$_{\omega_1}(S)[S]$ in which, for example, both \textbf{PPI} and \textbf{$\bsigma$} hold, without the need for large cardinals. Working in such a model, we can establish the following proposition, the conclusion of which was proved from PFA$(S)[S]$ in \cite{To} and asserted to be obtainable without large cardinals.

\begin{thm}
	If ZFC is consistent, it's consistent to additionally assume that locally compact, hereditarily normal, separable spaces are hereditarily Lindel\"of.
\end{thm}

\begin{proof}
	Let $X$ be such a space. By \ref{lem48} $X$ has countable spread. So does its one-point compactification $X^*$, which hence is countably tight \cite{A2}. If $X$ were not hereditarily Lindel\"of, it would include a right-separated subspace $\{x_\alpha:\alpha\in\omega_1\}$. Let $\{V_\alpha:\alpha\in\omega_1\}$ be open sets witnessing right-separation. Let $x_\alpha\in W_\alpha\subseteq\overline{W_\alpha}\subseteq V_\alpha$, with $W_\alpha$ open and $\overline{W_\alpha}$ compact. Applying $\bsigma$ to $X^*$, we see that $\{x_\alpha:\alpha\in\omega_1\}$ is $\sigma$-closed discrete in $W=\bigcup\{W_\alpha:\alpha\in\omega_1\}$. But $W$ is locally compact, separable, and hereditarily normal, so this contradicts \ref{lem48}.
\end{proof}

Also without large cardinals we obtain:

\begin{thm}\label{thm510}
	If ZFC is consistent, it is consistent to additionally assume that each hereditarily normal perfect pre-image of $\omega_1$ includes a copy of $\omega_1$.
\end{thm}

\begin{proof}
	Using $\bsigma$ and \textbf{PPI}, we can carry out the proof of Theorem \ref{thm34} above.
\end{proof}

We also have:
\begin{thm}\label{thm511}
	If ZFC is consistent, it is consistent to assume that every locally compact, first countable, hereditarily normal space with Lindel\"of number $\leq\aleph_1$ not including a copy of $\omega_1$ is paracompact.
\end{thm}

\begin{proof}
	We use the model of \ref{thm510}. In \cite{T} the second author asserted the following, but under PFA$(S)[S]$ instead of MM$(S)[S]$, which we now see should have been used.
	\begin{lem}\label{lem512}
		\textup{MM}$(S)[S]$ implies that if $X$ has Lindel\"of number $\leq\aleph_1$ and is locally compact, normal, and does not include a perfect pre-image of $\omega_1$, then $X$ is paracompact.
	\end{lem}
	
	In addition to the topological properties mentioned, the proof used $\bsigma$ and that the space was $\aleph_1$-collectionwise Hausdorff. For the purposes of \ref{thm511}, however, we get $\aleph_1$-collectionwise Hausdorff just from the Souslin forcing, since the space is first countable.
\end{proof}

MM$(S)[S]$ is also relevant for questions concerning the Baireness of $C_k(X)$, for locally compact $X$ (see \cite{GM,MN,T5}).

\begin{defn}
A \emph{\textbf{moving off} collection} for a space $X$ is a collection $\mc{K}$ of non-empty compact sets such that for each compact $L$, there is a $K \in \mc{K}$ disjoint from $L$.  A space satisfies the \textbf{Moving Off Property} (MOP) if each moving off collection includes an infinite subcollection with a discrete open expansion.
\end{defn}

\begin{defn}
$C_k(X)$, for a space $X$, is the collection of all continuous real-valued functions on $X$, considered as a subspace of the
compact-open topology on the Cartesian power $X^{\mathbb{R}}$.
\end{defn}

\begin{thm}[{~\cite{GM}}]
A locally compact space $X$ satisfies the MOP if and only if $C_k(X)$ is Baire, i.e. satisfies the Baire Category Theorem.
\end{thm}

\begin{lem}[{~\cite{GM,MN}}]\label{lem6}
Locally compact, paracompact spaces satisfy the MOP.
\end{lem}

\begin{thm}\label{thm35}
	$\mr{MM}(S)[S]$ implies that normal spaces satisfying the MOP are paracompact if they are:
	\begin{itemize}
		\item[(1)] locally compact, countably tight, and hereditarily normal, or
		\item[(2)] first countable and hereditarily normal, or
		\item[(3)] locally compact, countably tight with Lindel\"of number $\leq\aleph_1$, or 
		\item[(4)] first countable, with Lindel\"of number $\leq\aleph_1$, or
		\item[(5)] locally compact, countably tight, and countable sets have Lindel\"of closures.
	\end{itemize}
\end{thm}

\begin{proof}
	These all follow easily from \ref{thmConj2}, \ref{thm312}, and \textbf{Moore-Mr\'owka}, using:
	\begin{lem}[{ \cite{IN}}]
		In a sequential space, countably compact subspaces are closed.
	\end{lem}
	
	\begin{lem}[{ \cite{GM,MN}}]
		Countably compact spaces satisfying the MOP are compact.
	\end{lem}
	
	\begin{lem}[{ \cite{GM,MN}}]\label{lem510}
		First countable spaces satisfying the MOP are locally compact.
	\end{lem}
	
	\begin{lem}[{ \cite{B1}}]
		The one-point compactification of a locally compact space $X$ is countably tight if and only if $X$ does not include a perfect pre-image of $\omega_1$.
	\end{lem}
	
	If they have the MOP, sequential spaces do not include copies of $\omega_1$, so (1) follows from \ref{312}. (2) follows from (1) plus \ref{lem510}. (3) follows from \ref{lem512} plus \ref{thmConj2}. (4) follows from (3) plus \ref{lem510}. (5) follows from \ref{thm312}, \ref{thmConj2} and Balogh's Lemma above.
\end{proof}

In the special case of a space with the MOP, we have:
\begin{thm}\label{thm513}
	If ZFC is consistent, then it is consistent to additionally assume that first countable normal spaces satisfying the MOP and with Lindel\"of number $\leq\aleph_1$ are paracompact.
\end{thm}

\begin{proof}
	Such a space is locally compact and does not include a perfect pre-image of $\omega_1$.
\end{proof}

MA$_{\omega_1}$ gives counterexamples for the conclusions of \ref{thm35} and \ref{thm513}. See e.g. \cite{T5}.

\begin{thm}
	If ZFC is consistent, then it is consistent to assume that first countable hereditarily normal, locally connected spaces satisfying the MOP are paracompact.
\end{thm}

\begin{proof}
	The extra ingredient is that the local connectedness will enable us to decompose the space into a sum of pieces with Lindel\"of number $\leq\aleph_1$. More precisely,
	
	\begin{defn}
		A space $X$ is of \textbf{Type I} if $X=\bigcup\{X_\alpha:\alpha\in\omega_1\}$, where each $X_\alpha$ is open, $\alpha<\beta$ implies $\overline{X}_\alpha\subseteq X_\beta$, and each $X_\alpha$ is Lindel\"of.
	\end{defn}
	
	In \cite{T}, it is shown on page 104 that, assuming $\bsigma$ and hereditary $\aleph_1$-collectionwise Hausdorffness for a locally compact hereditarily normal space not including a perfect pre-image of $\omega_1$ that the closure of a Lindel\"of subspace is Lindel\"of. Then we quote:
	
	\begin{lem}[{ \cite{EN}}]
		If $X$ is locally compact, locally connected, and countably tight, then $X$ is a topological sum of Type I spaces if and only if every Lindel\"of subspace of $X$ has Lindel\"of closure.
	\end{lem}
	
	Since a topological sum of paracompact spaces is paracompact, this will complete the proof of the Theorem.
\end{proof}

\begin{prob}\textup{
	Without large cardinals, is there a model in which both $\bsigma$ and \textbf{LCN}$(\aleph_1)$ hold?}
\end{prob}

It may be of interest that \textbf{SRP} implies a weaker version of the conclusion of Theorem \ref{thm35}.2.

\begin{thm}
	\textup{\textbf{SRP}} implies every first countable, monotonically normal space satisfying the MOP is paracompact.
\end{thm}

\begin{lem}
	Suppose $S$ is a first countable stationary subspace of some regular cardinal. Then each $s\in S$ is an $\omega$-cofinal ordinal.
\end{lem}

\begin{proof}
	Each $s\in S$ is either isolated in $S$ or is a limit of some subset of $S$. By first countability, in the latter case, each such $s$ is a limit of a sequence of elements of $S$.
\end{proof}

\begin{proof}[Proof of Theorem.]
	Suppose not. Then by \cite{BR} the space includes a copy of a stationary subset of some regular cardinal. By \cite[37.18]{J} \textbf{SRP} implies that that stationary set includes a copy of a closed unbounded subset of $\omega_1$. That copy is closed, countably compact but not compact, contradicting the MOP.
\end{proof}

\begin{prob}[ \cite{GM}]\textup{
	Is there in ZFC a locally compact, normal, non-paracompact space with the MOP?}
\end{prob}

We conjecture th answer is positive. Large cardinals would be necessary to refute the existence of such a space, since an example can be constructed from the failure of the Covering Lemma for the Core Model K, which entails the consistency of measurable cardinals. We thank Peter Nyikos for referring us to \cite{G}, where that failure is used to construct a locally compact, locally countable, normal, non-paracompact space $X$ on $\kappa^+\times\omega_1$, where $\kappa^+$ is the successor of a singular strong limit cardinal of countable cofinality, such that the spaces $X_\alpha=\alpha\times\omega$ are metrizable for all $\alpha\in\kappa^+$. It follows that closed subspaces of $X$ of size $\leq 2^{\aleph_0}$ are locally compact and metrizable, so satisfy the MOP by \ref{lem6}. On the other hand,

\begin{thm}[ \cite{T5}]
	If a Hausdorff space $Z$ is locally countable, locally compact, and closed subspaces of $\leq 2^{\aleph_0}$ have the MOP, then $Z$ has the MOP.
\end{thm}

It follows that $X$ has the MOP.\qed\\

With MM$(S)[S]$ we have:

\begin{thm}
	\textup{MM}$(S)[S]$ implies that if $X$ is normal, locally compact, locally countable, and closed subspaces of size $\leq 2^{\aleph_0}$ are metrizable, then $X$ is metrizable.
\end{thm}

\begin{proof}
	By the preceding proof, $X$ has the MOP. By \ref{thm35}, to get that $X$ is paracompact, it suffices to show that countable subspaces of $X$ have Lindel\"of closures. But if $Y$ is a countable subset of $X$, $|\overline{Y}|\leq 2^{\aleph_0}$ and hence is separable metrizable and hence Lindel\"of. Once we have $X$ paracompact, it follows that $X$ is a topological sum of $\sigma$-compact subspaces. But each of these has size $\leq 2^{\aleph_0}$ and so is metrizable.
\end{proof}

\textbf{Axiom R} precludes stationary non-reflecting sets of $\omega$-cofinal ordinals in $\omega_s$, and hence the locally compact, $\aleph_1$-collectionwise Hausdorff ladder system space built on such a set; we can therefore ask:

\begin{prob}\textup{
	Does MM$(S)[S]$ imply \textbf{LCN}? Indeed, does MM imply locally compact $\aleph_1$-collectionwise Hausdorff spaces are collectionwise Hausdorff?}
\end{prob}

\section{Examples}

A question left open in \cite{LT1} is whether, as was shown for adjoining $\aleph_2$ Cohen subsets of $\omega_1$ in \cite{T1}, forcing with a Souslin tree would make normal spaces of character $\aleph_1$ $\aleph_1$-collectionwise Hausdorff. We shall show that the answer is negative by showing:

\begin{thm}
	\textup{MA}$_{\omega_1}(S)[S]$ implies that there is a normal non-$\aleph_1$-collectionwise Hausdorff space of character $\aleph_1$.
\end{thm}

\begin{proof}
	Let $S\subseteq 2^{<\omega_1}$ be a coherent Souslin tree. Fix a family ${\{a_s:s\in S\}\subseteq[\omega]^\omega}$ so that for $s<t\in S$, $a_t\subseteq^* a_s$ and for each $\gamma\in\omega_1$, $\{a_s:s\in S_\gamma\}$ is pairwise disjoint.
	
	For each limit $\delta\in\omega_1$, let $L_\delta\in\delta^\omega$ be a strictly increasing function with range cofinal in $\delta$ consisting of successor ordinals. For $a\subseteq\omega,$ let $L[a]=\{L_\delta(n):n\in a \}$. The generic $g$ for $S$ will enable us to define the required topology on the set $\omega_1$. We declare each successor ordinal to be isolated. For each limit $\delta$, the neighborhood filter for $\delta$ will be $\{L_\delta[a_s]\cup\{\delta\}:s\in g \}$. The set $C_0$ of limit ordinals is then a closed discrete set. By pressing down, we see that $C_0$ cannot be separated. It remains to show that the space is normal. It suffices to show that if $f$ is an $S$-name of a function from $C_0$ to $2$, then there is a neighborhood assignment $\{\dot{U}_\delta:\delta\in C_0\}$ and a cub $C_1$, such that for each $\alpha<\delta\in C_1$, $S$ forces that if $\dot{f}(\alpha)\neq\dot{f}(\delta)$, then $\dot{U}(\alpha)$ and $\dot{U}(\delta)$ are disjoint.
	
	There is a cub $C_1\subseteq C_0$ so that for all $\delta\in C_1$ and $\alpha<\delta_1$ each $s\in S_\delta$ decides the value of $\dot{f}(\alpha)$. For each $\delta\in C_0$, let $\delta^+$ denote the minimal element of $C_1$ above $\delta$, and choose a function $f_\delta:\omega\to 2$ so that for each $s\in S_{\delta^+}$ and each $n\in a_s$, $s$ forces $\dot{f}(\delta)=f_\delta(n)$. We will define an integer $n_\delta$ such that the value of $\dot{U}_\delta$ is forced by $s\in S_{\delta^+}$ to equal $\{\delta\}\cup L_{\delta}[a_s\setminus n_\delta]$. The sequence of functions $\{f_\delta:\delta\in C_0\}$ will be in the MA$_{\omega_1}(S)$ model.
	
	Let $\mathcal{Q}$ be the poset of partial functions $h$ from $\omega_1$ into $2$ such that ${h=^*\bigcup\{f_\delta\circ L_\delta^{-1}: \delta\in H \}}$, for some $H\in[C_0]^{<\omega}$. $\mathcal{Q}$ is ordered by extension. We claim that in ZFC, $\mathcal{Q}$ is ccc. If so, there will be a generic for $\aleph_1$ dense subsets of $\mathcal{Q}$ in a model of MA$_{\omega_1}(S)$. Let $\mathcal{H}=\{(h_\alpha,H_\alpha):\alpha\in\omega_1 \}$ be a subset of $\mathcal{Q}\times[C_0]^{<\omega}$, where $h_\alpha=\bigcup\{f_\delta\circ L_\delta^{-1}:\delta\in H_\alpha \}$. Choose any countable elementary submodel $M$ with $\mathcal{Q}$ and $\mathcal{H}$ in $M$. Let $\delta=M\cap\omega_1$ and $H_\delta\cap M=H$ and $H_\delta\setminus M=\{\delta_i:i<l\}$. We may assume that $\delta_0=\delta$ and then choose $\alpha_0\in M$ so that $H\subseteq\alpha_0$ and $L_{\delta_i}\cap\delta\subseteq\alpha_0$, for $0<i<l$. Notice that $h_\delta\!\!\upharpoonright\!\!\alpha$ is an element of $M$, for all $\alpha\in M$. In $M$, recursively choose $\alpha_0<\alpha_1<\cdots$ so that $h_{\alpha_{n+1}}\!\!\upharpoonright\!\!\alpha_n = h_\delta\!\!\upharpoonright\!\!\alpha_n$ and dom$(h_{\alpha_{n+1}})\subseteq\alpha_{n+2}$. With $\beta=\sup_n\alpha_n<\delta$, we have that there is an $n\in\omega$ such that $h_\delta\!\!\upharpoonright\!\!\beta = h_\delta\!\!\upharpoonright\!\!\alpha_n$. It follows that $h_\delta\!\!\upharpoonright\!\!\alpha_n\subseteq h_{\alpha_{n+1}}$, and so $h_\delta$ and $h_{\alpha_{n+1}}$ are compatible members of $\mathcal{Q}$.
	
	MA$_{\omega_1}(S)$ implies there is a generic for $\mathcal{Q}$ that adds a function $h$ from $\omega_1$ to $2$ that mod finite extends $f_\delta\circ L_\delta^{-1}$, for all $\delta\in C_0$. Now define $n_\delta$ to be chosen so that $h$ actually extends $f_\delta\circ L_\delta^{-1}[\omega\setminus n_\delta]$. Suppose $\alpha<\delta$, with $\delta\in C_1$, and let $s\in S_{\delta^+}$. Then $s$ forces that $f_\delta\circ L^{-1}_\delta=\dot{f}_\delta$ on $a_s$, and similarly, $s\!\!\upharpoonright\!\alpha^+$ forces that $f_\alpha\circ L^{-1}_\alpha=\dot{f}_\alpha$ on $a_{s\upharpoonright\alpha^+}$. Also, $h$ agrees with $f_\delta\circ L_\delta^{-1}$ on $a_s\setminus n_\delta$ and with $f_\alpha\circ L_\alpha^{-1}$ on $a_{s\upharpoonright\alpha^+}\setminus n_\alpha$. Thus if $\beta\in L_\delta[a_s\setminus n_\delta]\cap L_{\alpha}[a_{s\upharpoonright\alpha^+}\setminus n_\alpha]$, then $h(\beta)=\dot{f}(\alpha)=\dot{f}(\delta)$. This completes the proof that the space is normal.
\end{proof}

The strategy attempted in \cite{T3} was to expand a closed discrete
subspace of a locally compact normal space to a discrete collection of
compact $G_\delta$'s. There are limitations on such an approach, given
by the following example. 

\begin{thm}
	MA$_{\omega_1}(S)[S]$ implies there is a locally compact space of character $\aleph_1$ which includes a normalized closed discrete set which does not have a normalized discrete expansion by compact $G_\delta$'s.
\end{thm}

\begin{proof}
	We modify the previous example. Let $\mathcal{A}_s$ denote the Boolean subalgebra of $\mathcal{P}(\omega)$ generated by $[\omega]^{<\omega}\cup\{a_s:s\in S\}$. In the forcing extension by $S$, let $x_g$ denote the member of the Stone space $\mathcal{S}(\mathcal{A}_s/\text{FIN})$ containing $\{a_s:s\in g\}$.
	
	In the forcing extension, our space has the base set $(\omega_1\setminus C_0)\cup(C_0\times\mathcal{S}(\mathcal{A}_s))$. The points of $\omega_1\setminus C_0$ are isolated. For each $\delta\in C_0$ and $x\in\mathcal{S}(\mathcal{A}_s/\text{FIN})$, a neighborhood of $(\delta,x)$ must include $U_\delta(a)=L_\delta[a]\cup(\{\delta\}\times a^*)$ for some $a\in x$, where $a^* = \{p\in\mathcal{S}(\mathcal{A}_\mathcal{S}):a\in p\}$. Notice that $U_\delta(a)$ is disjoint from $\{\gamma\}\times\mathcal{S}(\mathcal{A}_s/\text{FIN})$, for all $\gamma\neq\delta$. It follows immediately that the sequence $D=\{(\gamma,x_g):\delta\in C_0\}$ is a closed discrete subset. It also follows from the proof of the normality of the previous example that $D$ is normalized.
	
	Now we show that $D$ does not have a normalized discrete expansion by compact $G_\delta$'s, indeed by any $G_\delta$'s. Assume that $\{\dot{Z}_\delta:\delta\in C_0\}$ is a sequence of $S$-names so that $\dot{Z}_\delta$ is forced to be a $G_\delta$ containing $(\delta,x_g)$. There is a cub $C_1$ such that for each $\alpha\in C_0$ and each $s\in S_{\alpha^+}$ (again, $\alpha^+$ is the minimal element of $C_1$ above $\alpha$), $s$ forces that $\dot{Z}_\alpha$ contains $\{\alpha\}\times a_s^*$. Since $S$ is ccc, the cub $C_1$ can be chosen to be a member of the PFA$(S)$ model.
	
	We use $C_1$ to define a partition of $C_0$: for each $\alpha\in C_0$, we define $\dot{f}(\alpha)$ to equal the value $g(\alpha^+)$ (i.e. the element of $S_{\alpha^+}$ that $g$ picks). Thus if $\delta$ is a limit of $C_1$ and $s\in S_\delta$, then $s$ forces a value for $\dot{f}\!\!\upharpoonright\!\!\delta$. Then a potential normalizing expansion would consist of a sequence $\{\dot{n}_\alpha:\alpha\in C_0\}$ of $S$-names of integers for which $L_\alpha[a_{g\upharpoonright\alpha^+}\setminus\dot{n}_\alpha]\cup(\{\alpha\}\times a_{g\upharpoonright\alpha^+}^*)$ is an open neighborhood of $\dot{Z}_\alpha$. There is a cub $C_2\subseteq C_1$ so that for each $\delta\in C_2$ and each $s\in S_\delta$, $s$ forces a value on $\dot{n}_\alpha$ for all $\alpha<\delta$. We may choose any $s_0\in g$ so that $s_0$ forces that $L_\alpha[a_{g\upharpoonright\alpha^+}\setminus\dot{n}_\alpha]\cap L_\delta[a_{g\upharpoonright\delta^+}\setminus\dot{n}_\delta]$ is empty whenever $\dot{f}(\alpha)\neq\dot{f}(\delta)$. Working in $V[g]$, we prove there is a stationary $E$ satisfying that $L_\delta[a_{g\upharpoonright\delta^+}]\cap\bigcup\{L_\alpha[a_{g\upharpoonright\alpha^+}\setminus\dot{n}_\alpha]:\alpha\in\delta\}$ is infinite, for all $\delta\in E$. If not, then there would be an assignment $\langle m_\delta:\delta\in C\rangle$ (for some cub $C$) so that $L_\delta[a_{g\upharpoonright\delta^+}\setminus m_\delta]$ would be disjoint from $\bigcup\{L_\alpha[a_{g\upharpoonright\alpha^+}\setminus\dot{n}_\alpha]:\alpha\in\delta\}$, for all $\delta\in C$. Pressing down, we would arrive at a contradiction.
	
	Let $\dot{E}$ denote the $S$-name of the stationary set whose existence was shown in the previous paragraph. Choose any $s$ above $s_0$ and any $\delta\in C_2$ such that $s$ forces that $\delta\in\dot{E}$. Without loss of generality, the height of $s$ is $\geq\delta^+$, but note that $s\!\!\upharpoonright\!\!\delta$ forces a value on $\dot{n}_\alpha$, for all $\alpha<\delta$. This means that $s\!\!\upharpoonright\!\!\delta^+$ forces that $\delta\in \dot{E}$, since it will also decide the value of $L_\delta[a_{g\upharpoonright\delta^+}]$. We also have that $s\!\!\upharpoonright\!\!\delta$ forces a value on $\dot{f}\!\!\upharpoonright\!\!\delta$ and so we can choose a value $e\in\{0,1\}$ so that $s\!\!\upharpoonright\!\!\delta$ forces that $L_\delta[a_{s\upharpoonright\delta^+}]$ intersected with $\{L_\alpha[a_{s\upharpoonright\alpha^+}\setminus\dot{n}_\alpha]:\alpha<\delta\text{ and }\dot{f}(\alpha)=e\}$ is infinite. We now have a contradiction, since $s\!\!\upharpoonright\!\!\delta^+\cup\{(\delta^+,1-e)\}$ forces that the assigned neighborhood of $\delta$ must meet the assigned neighborhood of $\alpha$, for some $\alpha<\delta$ with $\dot{f}(\alpha)=e\neq\dot{f}(\delta)$.
\end{proof}

\section{Point-countable type}

There is another normal-implies-collectionwise-Hausdorff result holding in $L$ for which we don't know whether it holds in our MM$(S)[S]$ model:

\begin{defn}
	A space is of \textbf{point-countable type} if each point is a member of a compact subspace which has a countable outer neighbourhood base.
\end{defn}

Spaces of point-countable type simultaneously generalize locally compact and first countable spaces, and V$=$L implies normal spaces of point-countable type are collectionwise Hausdorff \cite{W}.

\begin{prob}
	Does MM$(S)[S]$ imply normal spaces of point-countable type are $\aleph_1$-collectionwise Hausdorff?
\end{prob}

The usual arguments would show that if so, in our front-loaded model of MM$(S)[S]$, normal spaces of point-countable type would be collectionwise Hausdorff.

\vspace{1cm}

\textbf{Acknowledgement.} We thank Peter Nyikos for catching errors in an earlier version of this manuscript.

\nocite{*}
\bibliographystyle{acm}
\bibliography{normality}

{\rm Alan Dow, Department of Mathematics and Statistics, University of North Carolina,
Charlotte, North Carolina 28223}

{\it e-mail address:} {\rm adow@uncc.edu}

{\rm Franklin D. Tall, Department of Mathematics, University of Toronto, Toronto, Ontario M5S 2E4, CANADA}

{\it e-mail address:} {\rm f.tall@math.utoronto.ca}

\end{document}